\documentclass[smallextended]{svjour3}       
\smartqed  
\usepackage{graphicx}
\journalname{Numer.\ Math.}
\def\pput(#1,#2)#3{\noindent\smash{\raise#2pt\hbox to 0pt
   {\kern #1pt #3\hss}}\ignorespaces}
\def\Im{\hbox{\rm Im\kern .7pt}}
\def\Re{\hbox{\rm Re\kern .7pt}}
\def\arg{\hbox{\rm arg\kern .2pt}}
\def\twopi{2\kern.2 pt\pi}
\def\Gex{\Gamma_{\scriptsize\rm exit}}

\title{Representation of conformal maps by rational functions}

\author{Abinand Gopal \and Lloyd N. Trefethen}

\institute{A. Gopal\at
Mathematical Institute\\
University of Oxford\\
Oxford OX2 6GG, UK\\
\email{gopal@maths.ox.ac.uk} \\
\and
L. N. Trefethen \at
Mathematical Institute\\
University of Oxford\\
Oxford OX2 6GG, UK\\
\email{trefethen@maths.ox.ac.uk} \\
}

\date{Received: date / Accepted: date}

\usepackage{graphics,graphicx}
\usepackage{upquote,verbatim,amssymb}

\begin{document}

\maketitle


\begin{abstract}
The traditional view in numerical conformal mapping is that
once the boundary correspondence function has been found, the
map and its inverse can be evaluated by contour integrals.
We propose that it is much simpler, and 10--1000 times faster, to
represent the maps by rational functions computed by the
AAA algorithm.  To justify this claim, first
we prove a theorem establishing root-exponential convergence
of rational approximations
near corners in a conformal map, generalizing a result of
D. J. Newman in 1964.
This leads to the new algorithm for approximating
conformal maps of polygons.  Then we turn to smooth
domains and prove a sequence of four
theorems establishing that in any conformal
map of the unit circle onto a region with a long and slender part,
there must be a singularity or loss of univalence
exponentially close to the boundary, and polynomial approximations
cannot be accurate unless of exponentially high degree.  This 
motivates the application of the new algorithm to smooth
domains, where it is again found to be highly effective.
\end{abstract}

\section{Introduction}
\label{sec-intro}
Traditionally in numerical conformal mapping, one first
determines the boundary correspondence function, the
one-to-one mapping between the boundaries of the domain
and the image regions, and then one evaluates the map
or its inverse in the interior by methods such as Cauchy
integrals~\cite{cmft,gaier,hqr,henrici,wegmann}.  If the
domains are smooth, the first step is often carried out
by solving an integral equation.  Most domains of interest
in practice are not smooth, and the corners that appear so
often bring in new challenges that traditionally would be
addressed by the exploitation of the local structure of each
singularity~\cite{bds,lehman,pwh}.  In the extreme case of the
map of a disk or a half-plane onto a polygon, the problem is
``all singularities,'' and here one may make use of a numerical
implementation of the Schwarz--Christoffel (henceforth SC)
formula, for which the standard software is Driscoll's SC
Toolbox in MATLAB~\cite{toolbox,SCbook,papstyl}.

In this paper we propose a different approach to the second phase
of computing a conformal map.  Instead of evaluating Cauchy or SC
integrals, we propose the use of rational functions to represent
the maps, one rational function in the forward direction and
another in the inverse direction.  Instead of exploiting the
structure of singularities, we ignore their presence and let the
rational approximations do the work.  This is not asymptotically
optimal for very high accuracies, but in the practical regime,
it is remarkable how efficient it can be.  For a polygonal region
with around six corners, we find that rational functions of degree
30 to 100 typically suffice for mapping to 7-digit accuracy,
and the evaluation time is on the order of a microsecond per
point mapped on a laptop.  Mapping tens of thousands of points
back and forth accordingly takes a fraction of a second.

Until recently, no fast method was available for constructing
such rational approximations, but this situation has changed with
the appearance of the AAA algorithm~\cite{AAA}.  This algorithm
represents rational functions stably by quotients of partial
fractions rather than polynomials, with the necessary support
points determined adaptively.\footnote{The name comes from
``adaptive Antoulas--Anderson'' and is pronounced ``triple-A.''}
If $Z$ and $F$ are sets of, say, 5000 corresponding points on the
boundaries of the domain and target regions, the command {\tt r =
aaa(F,Z)} in Chebfun~\cite{chebfun} will typically construct a
rational function mapping $Z$ to $F$ in less than a second on
a laptop.  A {\tt aaa} code in MATLAB is listed in~\cite{AAA},
and a Python implementation is available from the first author.

The mathematical basis of our method is a new result, proved in
Section 2, that generalizes a theorem published by Donald Newman in
1964~\cite{newman}.\ \ Newman
showed that the minimax error $E_n$ for
approximation of $f(x) = |x|$ on $[-1,1]$ by type $(n,n)$ rational
functions decreases at the rate $\exp(-A\sqrt n\kern .8pt)$ for some
$A>0$ as $n\to\infty$.  The sharp constant is $A = \pi$, and more
generally, for approximation on $[-1,1]$ of $f(x) = |x|^\alpha$ with
$\alpha>0$, Stahl~\cite{stahl03} showed that root-exponential
convergence occurs at the rate
\begin{equation}
E_n = O(e^{-\pi \sqrt{\alpha n}}\kern 1.5pt).
\label{stahl}
\end{equation}
This estimate is sharp unless $\alpha$ is an even integer,
in which case $f$ reduces to a polynomial. In the context
of our conformal maps, the essential problem is that of
approximating the complex function $x^\alpha$ on a one-sided
complex neighborhood of $[-1,1]$ rather than the real function
$|x|^\alpha$ on $[-1,1]$ alone.  In Section~2, applying a
different method from Newman's, we prove that the convergence
rate for this problem is also root-exponential.  Section~3 then
applies this result to conformal maps of polygons.

In Section~4 we turn to regions with analytic boundaries.  We start
from a theoretical perspective, showing in a sequence of
four theorems that because of the
``crowding phenomenon'' of conformal mapping, the conformal map
of a region with a smooth boundary necessarily has a singularity
or loss of univalence exponentially close to the boundary if the
region contains a part that is long and narrow in a certain sense.
Under the same assumptions, polynomial approximations are proved
to be inaccurate unless their degree is exponentially large.
We believe these theorems are a significant contribution to the
literature of crowding.  In the context of the present paper,
their role is to
motivate Section~5, which applies our rational function
method to conformal maps of domains with analytic boundaries.
Section~\ref{numerics} examines the accuracy of our approximations.

One of the consequences of our rational approximations
is that they eliminate much of
the practical difference between conformal mapping methods
in the ``forward'' and ``inverse'' directions.  For example,
the Schwarz--Christoffel formula maps from a half-plane or a disk
to the problem domain, which traditionally introduces an asymmetry
in the cost of evaluating the map in the two directions.  Similarly,
there have been computational consequences of the fact 
that integral equations from the problem domain to the disk
are often linear, whereas in the direction from the disk to the
problem domain they are usually nonlinear.
With the techniques introduced here, the
directionality of the underlying representation becomes less important, since,
once the boundary correspondence map has been computed, the
initial representation is discarded in favor of
rational functions in the two directions.

\section{Rational approximation of \boldmath$x^\alpha$}

In this section we prove a result that is the basis of our
algorithm, establishing that rational approximations can be
highly efficient at representing a conformal map near a corner.
The local behavior of a conformal map with straight sides at a
corner of angle $\alpha\kern .3pt \pi$ is of type $x^\alpha$ (or
$x^{1/\alpha}$ in the inverse direction), as can be established
from the Schwarz reflection principle~\cite{SCbook}.

Let $H$ denote the closed upper half of the unit disk and
$\|\cdot\|_H$ the supremum norm on $H$.  In fact, the only
property of $H$ we shall use is that it is a bounded subset of
the upper half-plane.  Our proof is adapted from the argument
for approximation of $|x|$ on $[-1,1]$ given on pp.~211--212
of~\cite{atap}.

\smallskip

\begin{theorem}
Let $\alpha$ be any positive number.  There exist a
constant\/ $A>0$ and type $(n,n)$ rational approximations\/ $r_n$ such
that as $n\to \infty$,
\begin{equation}
\|x^\alpha - r_n \|_H^{} = O(e^{-A \sqrt n}).
\label{a1}
\end{equation}
\end{theorem}

\smallskip

\begin{proof}
If $r_n(x)\approx x^\alpha$ with $\alpha \in (0,1]$, then $x^k
r_n(x) \approx x^{k+\alpha}$, so without loss of generality we may
assume $\alpha\in (0,1]$.  Moreover, if $r_n(x)\approx x^\alpha$
with $\alpha \in (0,1/2\kern .3pt]$, then $(r_n(x))^2 \approx
x^{2\alpha}$, so without loss of generality we may further assume
$\alpha\in (0,1/2\kern .3pt]$.

We start from the identity
\begin{equation}
x^\alpha  = C \int_0^\infty {x\kern .8pt dt\over t^{1/\alpha } + x}, \quad
C = {\sin(\alpha\pi)\over \alpha\pi},
\label{a2}
\end{equation}
valid for $x\not\in (-\infty,0\kern .3pt]$,
which can be derived via the substitution $u  = t/x^\alpha$ and
the integral $\int_0^\infty du/(u^b + 1) =
\pi\csc(\pi/b)/b$ for $b>1$~\cite[eq.~3.241.2]{grad}.
Since we wish to apply (\ref{a2}) for $x$ in the
closed upper half-plane, we rotate the integration contour to
exclude zeros of the denominator,
leaving the value of the integral unchanged since
the integrand decreases faster than linearly as $|t|\to\infty$.\ \ Specifically,
we make the change of variables $t =
e^{\alpha \pi i/2 + s}$, $dt = e^{\alpha \pi i/2+ s} ds$, where $s$
is real, which converts (\ref{a2}) to
\begin{equation}
x^\alpha  = C \int_{-\infty}^\infty {x\kern .8pt e^{\alpha \pi i/2 + s} ds\over
e^{\pi i/2 + s/\alpha }+x}.
\label{a3}
\end{equation}
Note that the integrand decays exponentially at the rate
$e^{s(1-1/\alpha )}$ as $s\to +\infty$, and this is uniformly
true for all values of $x\in H$ since $H$ is bounded.  As for $s
\to -\infty$, here the integrand decays exponentially at the rate
$e^s$ for each $x\ne 0$, uniformly for all $x$.

To get a rational approximation to $x^\alpha$, we now
approximate (\ref{a3}) by the trapezoidal rule with node
spacing $h>0$:
\begin{equation}
r(x) = h \kern .8pt C \sum_{k = -(n-1)/2}^{(n-1)/2}
{x\kern .8pt e^{\alpha \pi i/2 + kh}\over
e^{\pi i/2 + kh/\alpha } + x}.
\label{a4}
\end{equation}
Here $n$ is a positive even number, and there are $n$ terms in the
sum, so $r(x)$ is a rational function of $x$ of type $(n,n)$.  Its
poles are on the negative imaginary axis,
so $r$ is analytic in the upper half-plane.

As reviewed in~\cite{trap}, the error $|r(x)-x^\alpha|$ in the
trapezoidal rule approximation can be decomposed into two parts.
One part is introduced by terminating the sum at $n<\infty$,
on the order of $e^{-nh/2}$.  (By ``on the order,'' we mean
that the dominant exponential term is correct; there may
be further lower-order algebraic factors.)  The other part
is introduced by the finite step size $h>0$.  According to
Theorem~5.1 of~\cite{trap}, this will be of order $e^{-2\pi
d/h}$, where $d$ is the half-width of the strip of analyticity
of the integrand around the real $s$-axis.  To determine
this half-width, we note that the denominator of (\ref{a4})
will be 0 when $e^{\pi i/2 + s/\alpha }$ is equal to $-x$,
and for $x$ in the upper half-plane, this can only happen when
the argument of $e^{s/\alpha }$, namely the imaginary part of
$s/\alpha $, is at least as large as $\pi/2$ in absolute value.
The half-width of the strip of analyticity is consequently $d =
\pi \alpha /2$, giving an error of $e^{-\alpha \pi^2/h}$.

We now choose $h$ to balance the errors $e^{-nh/2}$ and
$e^{-\alpha \pi^2/h}$.
The balance occurs with $h = \pi \sqrt{2\alpha/n}$, giving a
convergence rate
\begin{equation}
\|r_n(x) - x^\alpha \|_H^{} \lesssim \exp(-\pi\sqrt{\alpha n/2}\kern .8pt).
\label{a5}
\end{equation}
We use the inexact symbol ``$\,\lesssim\,$''
since we have only tracked the exponential
term, not lower-order algebraic factors, but this is enough
to establish (\ref{a1}) for any value $A<\pi\sqrt{\alpha/2}$.
\end{proof}

\smallskip

The constants in our argument are not optimal, and one reason
is that the approximation (\ref{a4}) is valid (nonuniformly)
throughout the upper half-plane, not just in $H$.  We do not
know the optimal constants, which appear to be different from
those given in (\ref{stahl}) for approximation of $|x|^\alpha$
on $[-1,1]$.  Our convergence rate bound slows down to zero
as $\alpha \to 0$, but this is also the case in (\ref{stahl}).

Note that the poles of the approximation (\ref{a4}) cluster
exponentially near $x=0$, with exponentially decreasing
residues, as we shall observe in our numerical experiments;
see in particular Figure~\ref{fig1_closeup}.

Though a single straight-sided corner is the case we have
analyzed, in practice the root-exponential approximation
effect is more general.  If there are $k$ singular corners,
the type $(n,n)$ of the rational function required for a given
accuracy increases only approximately in proportion to~$k$,
and whether the sides are straight or not is immaterial.
For example, the standard branch of $z^{1/2}\log(z)$ is as
easily approximated as that of $z^{1/2}$ for $z\approx 0$
in the upper half-plane.  Application of {\tt aaa} for these
functions gives approximations accurate to 6 digits on $[-1,1]$
with $n = 36$ and $29$, respectively.

\section{Conformal maps of polygons}

\begin{figure}
\begin{center}
\vskip .2in
\includegraphics[scale=.792]{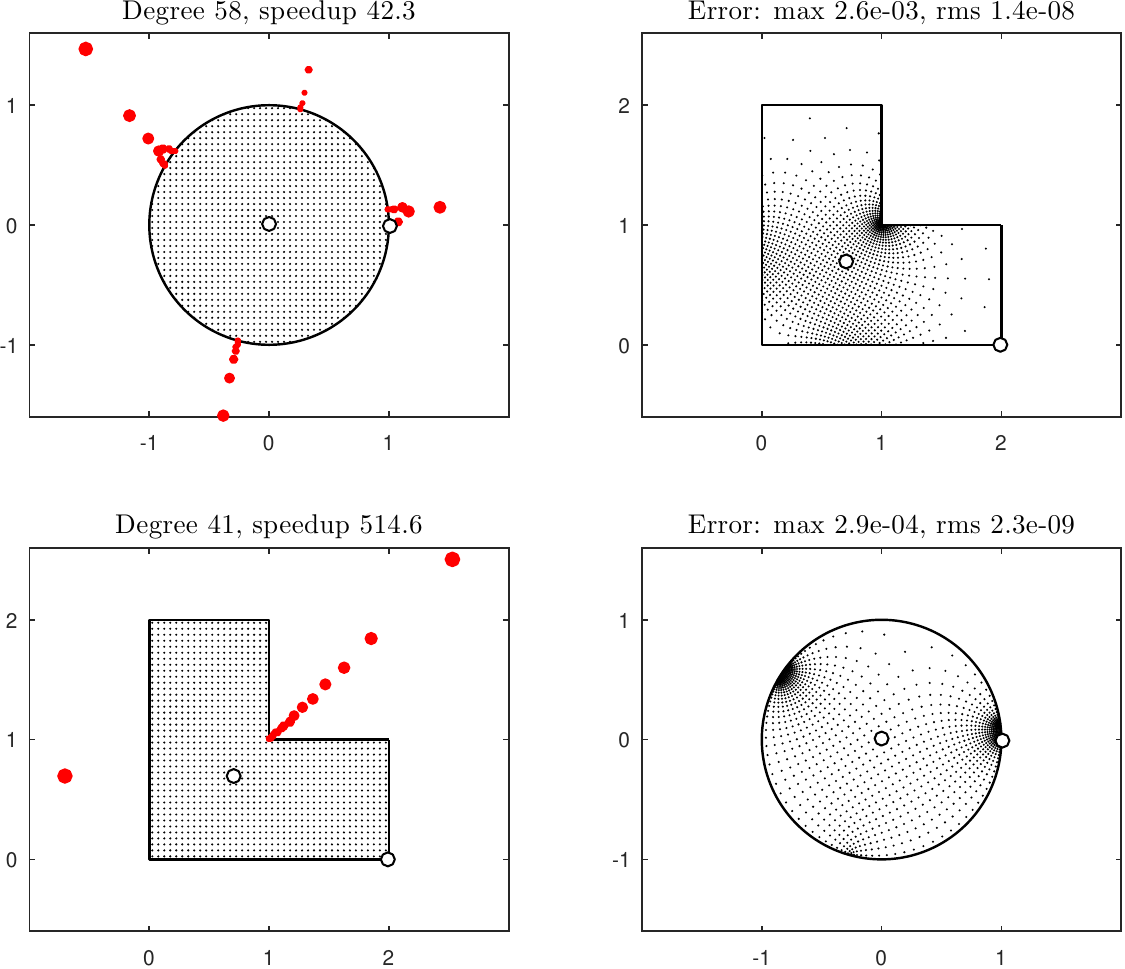}
\end{center}
\caption{\label{fig1_L} Conformal map $f$ of the unit disk onto an
L-shaped region (top row) and the inverse map $f^{-1}$ (bottom row).
The parameters for the map are computed by the SC Toolbox and then
$f$ and $f^{-1}$ are represented compactly to about eight digits of
accuracy (except very near the vertices)
by rational functions $r\approx f$ and $s\approx f^{-1}$
computed by the AAA algorithm.  The white dots show the
normalization of the map by an interior and a boundary
point, and the small black dots in the interior
show\/ about $1000$ points on a regular grid and their images.  The large red dots
show poles of the rational functions sized according to the absolute
values of their residues.}
\end{figure}

Let $\Delta$ be the closed unit disk in the $z$-plane and
$P$ the closed region bounded by a polygon with $k$ vertices
$w_1,\dots , w_k$ in the $w$-plane.  The SC formula represents
a conformal map $f: \Delta \to P$ in terms of an integral with
fractional power singularities at the {\em prevertices} $z_j
= f^{-1}(w_j)$~\cite{SCbook}.  Determining the prevertices
for a given $P$ is a numerical problem known as the {\em SC
parameter problem.}

We illustrate our method by applying it to five SC maps
computed by the Schwarz--Christoffel Toolbox~\cite{toolbox}.
Our first example is a map of the unit disk onto
an L-shaped region, shown in
Figure~\ref{fig1_L}.\footnote{We use the SC Toolbox commands
{\tt w = [2+i 1+i 1+2i 2i 0 2];}
{\tt c = .7+.7i;}
{\tt p = polygon(w);}
{\tt opts = sctool.scmapopt('Tolerance',1e-12);}
{\tt f = diskmap(p,opts);}
{\tt f = center(f,c)}.}
The Toolbox computes the forward and inverse
maps with approximately twelve digits of accuracy almost everywhere
in the domain, except that very close to some of the corners or
their preimages, the accuracy falls to five or six digits (see
Section~\ref{numerics}).\footnote{Following Driscoll's suggestion
(private communication), we have also improved the accuracy of
the inverse map by adding the line {\tt newton = false} after the
command {\tt [ode,newton,tol,maxiter] = sctool.scinvopt(options)}
in the Toolbox file {\tt @diskmap/private/dinvmap.m}.}
Figure~\ref{fig1_L} displays
this conformal map and its inverse and illustrates the
rational approximations that are the subject of this paper.
The plots in the top row show the forward map $f$.
First, the mapping parameters are computed with the SC Toolbox as
above.  Then 3000 equispaced points on the unit circle are collected
in a vector $Z$.  (We take the number of points to be
$500\kern .5pt k$, i.e., 500 times the
number of vertices; see Section~\ref{numerics}.)
The vector of images $F = f(Z)$ is computed
with the SC Toolbox.  We then execute the Chebfun commands
\begin{verbatim}

    [r,rpol,rres] = aaa(F,Z,'tol',1e-7,'mmax',200)
    [s,spol,sres] = aaa(Z,F,'tol',1e-7,'mmax',200)
\end{verbatim}

\noindent to determine forward and inverse rational functions
$r$ and $s$ such that $F \approx r(Z)$ and $Z \approx s(F)$;
these computations take less than a second.  The poles of $r$
and~$s$ are returned in the vectors {\tt rpol} and {\tt spol},
computed from the AAA barycentric representation by solving a
generalized eigenvalue problem~\cite{AAA}, and the corresponding
residues are returned in {\tt rres} and {\tt sres}.

To make the plots, we next apply $r$ to map a grid of points in
$\Delta$ to their images in $P$, which takes about a millisecond.
As usual with conformal maps, there are exponential distortions
involved (see the next section), leading to a very uneven
distribution of image dots.  The second row of the figure
shows the inverse map $f^{-1}$, including a grid of points
in $P$ and their equally unevenly distributed (pre)images in
$\Delta$ computed with $s$, again requiring about a millisecond.
The plots also show a white dot in the interior and another
on the boundary of each region, together with their images,
to complete the specification of the conformal mapping problem.

\begin{figure}
\begin{center}
\vskip .2in
\includegraphics[scale=.792]{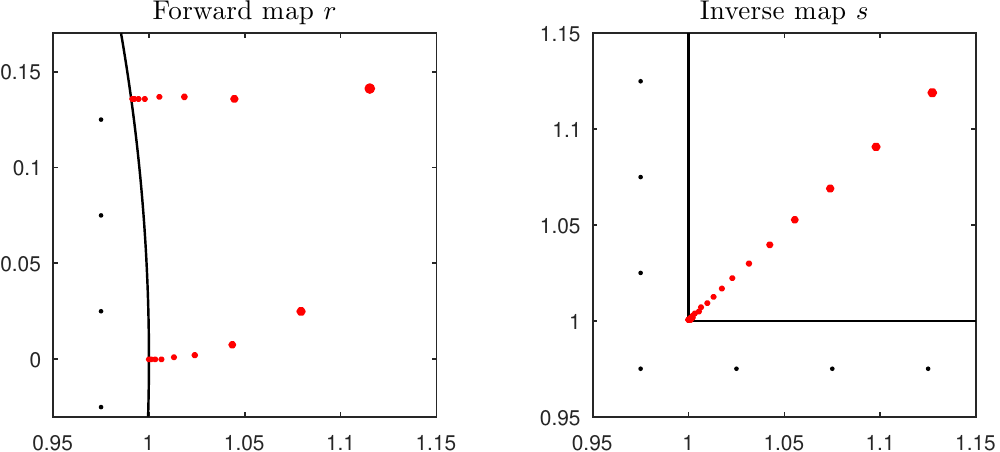}
\end{center}
\caption{\label{fig1_closeup} Closeups of Figure\/~$\ref{fig1_L}$.
The poles approach the singularities with exponentially decreasing
spacing and residues, as one would expect based on Theorem~1.}
\end{figure}

Most striking in Figure~\ref{fig1_L} are the red dots
representing poles of $r$ (top left) and~$s$ (bottom left).
These are sized to indicate the magnitude of each residue, with
a logarithmic spacing from the smallest diameter for magnitude
$10^{-3}$ or less to the largest for magnitude $1$ or more.
In the upper-left plot, we see chains of poles of\/~$r$ with
diminishing residues approaching all six prevertices (two pairs
of which are quite close together).  This is a familiar effect in
rational approximation~\cite{stahl,suetin}, and it occurs because
all six points correspond to branch point singularities of $f$,
with local behavior of type $(z-z_j)^{1/2}$ for the five salient
corners and $(z-z_j)^{3/2}$ at the reentrant corner.  The degree
58 marked in the title indicates that $r$ is a rational function
of type $(58,58)$, that is, with 58 poles (some of which are not
in the picture).  The lower-left plot, on the other hand, shows
poles only approaching the reentrant corner for the inverse map.
This is because the other corners have local behavior of type
$(w-w_j)^2$, which is not singular.\footnote{One can prove
that $f^{-1}$ is analytic at a right-angle salient corner by
analytically continuing the conformal map around the vertex
with four applications of the Schwarz reflection principle.
Such an argument shows that in general, a corner of a polygon
is a nonsingular point of the inverse conformal map if and only
if the interior angle is $\pi$ divided by an integer.  This is
essentially the same as the observation about sharpness just
after (\ref{stahl}).}  Figure~\ref{fig1_closeup} shows closeups from
each of these figures.

The ``speedup'' numbers in the titles give estimates of how
much faster AAA is for evaluating each map than the SC Toolbox.
To compute these numbers, timings for mapping the grid of
points are compared with timings for the corresponding
SC Toolbox command.  Throughout our explorations, we have found
that the rational functions are typically 10--1000 times faster
than the Toolbox.  The speedup numbers in our plots are only
approximate; they vary from one run to the next.

\begin{figure}
\begin{center}
\vskip .2in
\includegraphics[scale=.792]{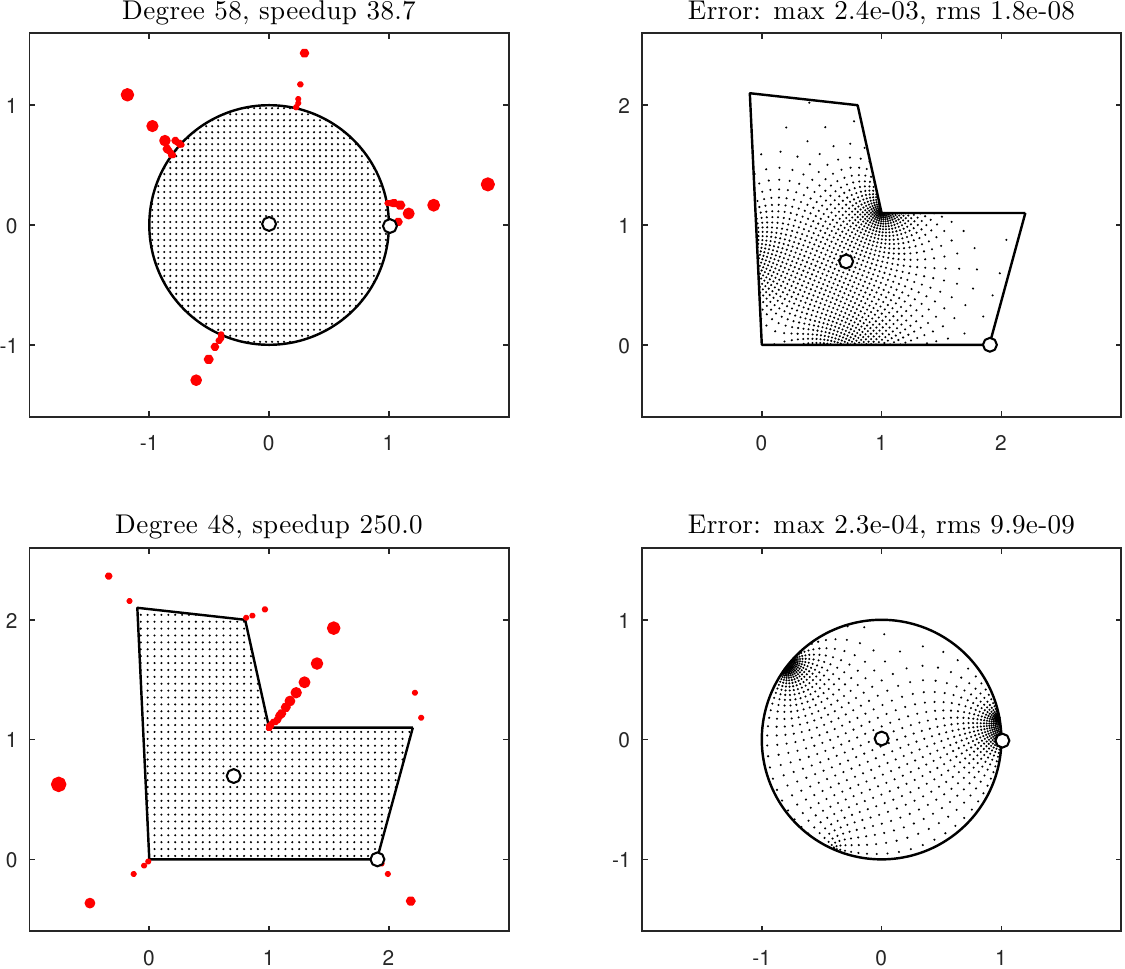}
\end{center}
\caption{\label{fig2_Lpert} A perturbed L-shaped region.  All the
corners are now singular, and\/ $r$ and\/ $s$ have poles near
all six prevertices/vertices.  However, some singularities are much
stronger than others.}
\end{figure}

Finally there are the errors listed in the titles.  In the
upper-right plot the first of these numbers, labeled `max',
is determined by sampling the boundary on a grid four times
finer than the sample set and computing the maximum of $|f(z) -
r(z)|$ at these points, with $f$ calculated by the SC Toolbox.
Since $f$ and $r$ are both analytic in $\Delta$, the maximum
modulus principle implies that this can be counted on as
close to a true maximum error in the rational representation.
These numbers are generally disappointing, but this reflects
diminished accuracy only in extremely small regions close to
the vertices, as we shall discuss in Section~\ref{numerics}
(see in particular Figure~\ref{contplot}).\ \ Elsewhere, the
accuracy is excellent, and this is reflected in the ``rms''
errors, which show root-mean-square errors in the grid of points.
Note that the rms errors are much smaller than $1/\sqrt{1000}$
times the max errors, implying that none of the ${\approx}\kern
1pt 1000$ grid points have errors even close to maximal.

\begin{figure}
\begin{center}
\vskip .2in
\includegraphics[scale=.792]{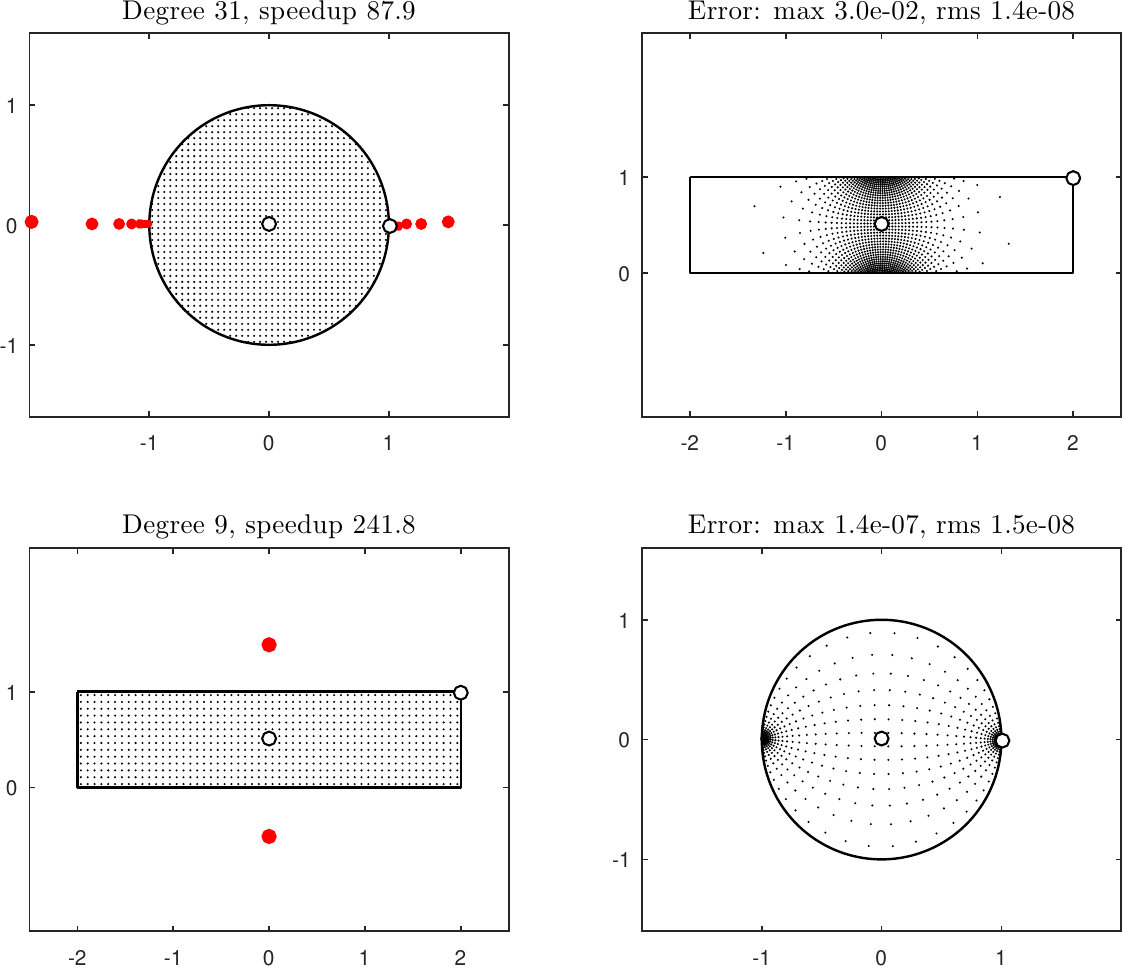}
\end{center}
\caption{\label{rectangle} Conformal map of a $4$-to-$1$ rectangle, 
an example we shall discuss further in the next section.  Here
$f^{-1}$ is an elliptic function, analytic in
a neighborhood of $P$.  On the unit circle, the pairs of
prevertices are too close to distinguish by eye.}
\end{figure}

The lower-right plot of Figure~\ref{fig1_L}, for the inverse map,
is labeled in the same manner, now with the maximum error based
on the maximum of $|f^{-1}(w) - s(w)|$, again on a boundary grid
four times finer and with $f^{-1}$ computed by the SC Toolbox.

We mentioned that the five salient right angles of the L-shaped
region are nonsingular.  Figure~\ref{fig2_Lpert} highlights
this effect by repeating Figure~\ref{fig1_L} with the vertices
perturbed.  Now all the corners are singular, and poles of $s$
appear near all of them.  At the same time, the widely varying
numbers and sizes of the dots illustrate how much more important
some singularities are than others.  The rational function $s$
has just three poles within a distance $1/2$ of the singular
vertex near $1+2\kern .5pt i$, for example, with residues of
absolute value less than $10^{-3}$.  By contrast there are 23
poles within the same distance of the reentrant corner near
$1+i$, their distances decreasing from $O(1)$ to $O(10^{-4})$
as the absolute values of the residues decrease from $O(10^{-1})$
to $O(10^{-7})$.

At the other extreme, Figure~\ref{rectangle} shows the conformal
map onto a rectangle.  Here the inverse map has no singular
corners, and indeed it is a doubly periodic elliptic function (as
follows from the Schwarz reflection principle).  In this case, no
poles approach the boundary of $P$, and the associated rational
function has degree just~$9$.  The two poles visible on these
axes lie at about $1.5\kern .5pt i + 10^{-7}(4+4\kern .5pt i)$
and $-0.5\kern .5pt i + 10^{-7}(3+2\kern .5pt i)$, very close
to the positions $1.5\kern .5pt i$ and $-0.5\kern .5pt i$ of the
innermost poles of the elliptic function.  Their residues match
the theoretical value $-2/\pi$ to about three digits.  In this
as in most rational approximations, one should bear in mind
that $r$ or $s$ will match poles or residues of $f$ or $f^{-1}$
only ``as a means to the end'' of approximating $f$ or $f^{-1}$
itself at the prescribed sample points, so not too much should be
read into the number of digits of agreement.  For example, this
elliptic function has infinitely many poles in the plane, and
obviously $s$ is not going to approximate all of them.  One may
also note in the upper-left image of Figure~\ref{rectangle},
as in several of our other figures, that the AAA algorithm does
not respect symmetry of approximation problems.  One might have
expected the poles of $\kern 1pt r$ to be real, or at least
to fall in complex conjugate pairs, but the algorithm does not
impose this condition.  (As mentioned in~\cite{AAA}, one could
develop variant algorithms with this property.)

\begin{figure}
\begin{center}
\vskip .2in
\includegraphics[scale=.792]{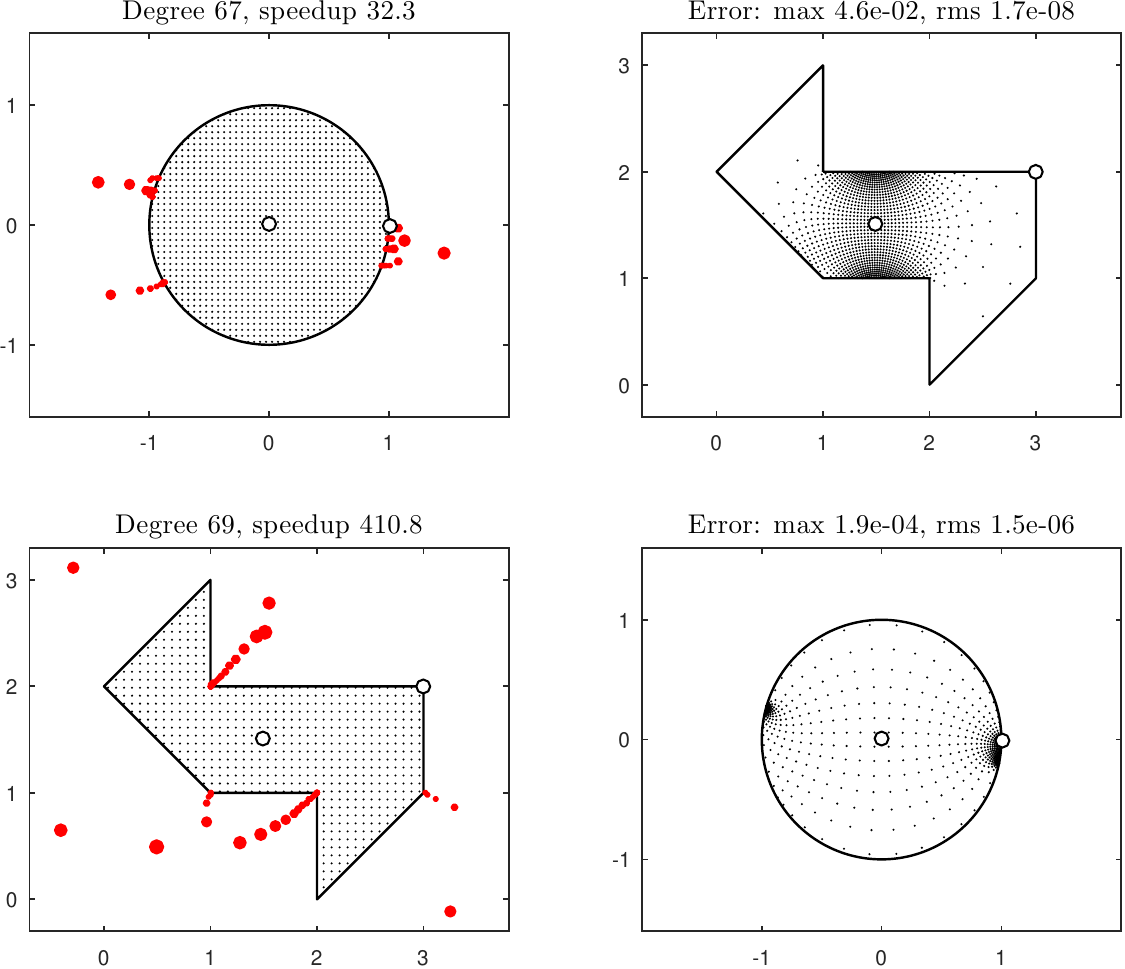}
\end{center}
\caption{\label{fig4_iso} Conformal map of an isospectral drum from
Gordon, Webb, and Wolpert\/~{\rm \cite{gww}}.}
\end{figure}

Figure~\ref{fig4_iso} shows one of the isospectral
drums made famous by Gordon, Webb, and Wolpert~\cite{gww}.
Again it is clear from the pole distribution
that there are particularly strong singularities at the reentrant corners.

\begin{figure}
\begin{center}
\vskip .2in
\includegraphics[scale=.792]{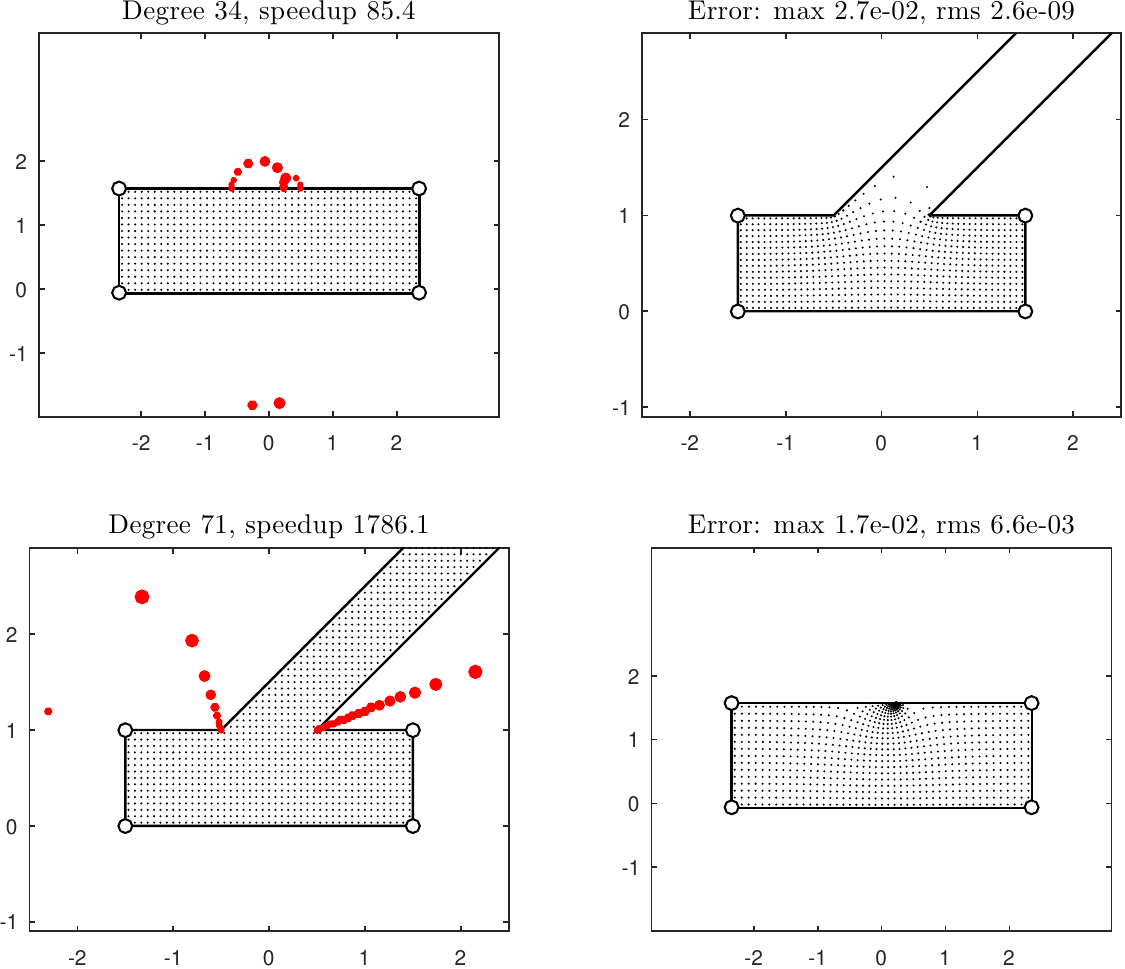}
\end{center}
\caption{\label{rmap} Map of a rectangle onto an infinite polygon,
illustrating that rational approximations are not limited to bounded
regions.  The black dots in the second row were computed as a grid of
points in the portion of the polygon with
imaginary part ${\le}\kern 1pt 3$, though the figure would look the
same if such points could be included all the way out to $\infty$.}
\end{figure}

These four examples involve maps of the unit disk to a polygon,
but there are other possibilities.  For example, variants of
the SC formula are implemented in the SC Toolbox for mapping
to a polygon from a half-plane, a rectangle, or an infinite
strip~\cite{toolbox,SCbook}.  Moreover, the domains need not be
bounded, and if they are unbounded, rational functions can still
be used to represent the conformal maps.  Figure~\ref{rmap}
illustrates this with images for the map of a rectangle to an
infinite polygon $P$ consisting of another rectangle along one
side of which a channel extending to $\infty$ has been attached
at a $45^\circ$ angle.  The polygon $P$ is first prescribed;
then a rectangle in the $z$-plane is found with the uniquely
determined aspect ratio such that its four corners can map to
the four right-angle corners of $P$.

\section{Exponential distortions with analytic boundaries}
We now turn to smooth domains, specifically, domains bounded by
analytic Jordan curves.   If $f$ is a conformal map of the unit
disk onto such a domain, then $f$ can be analytically continued
to a neighborhood of the closed unit disk.  This implies that $f$
can be approximated by degree $n$ polynomials with exponential
convergence, that is, errors diminishing at a rate $O(\kern .5pt
\rho^{-n})$ as $n\to \infty$ for some $\rho > 1$.  This might
seem to suggest that approximation by rational functions should
not be needed.

In fact, that conclusion would be mistaken, for the constant
$\rho$ is likely to be exponentially close to $1$.  The aim of
this section is to establish some theorems in this direction,
whose root is the ``crowding phenomenon'' of conformal
mapping~\cite{banjai,dp,papstyl,sc}.  The crowding effect can be
seen in all the plots of conformal maps we have presented, where in every
case, a regular array of points in one domain maps to a set of
points with exponentially varying densities in the other. 
In Figure~\ref{rectangle}, for example, the distortion is
extreme enough that as far as one can tell from the clusters
of poles, four prevertices on the unit circle appear to be just two.
Further examples will be given in Section~5.

Let $f$ be a conformal map of the unit disk $D$ in the $z$-plane
onto a region $\Omega$ in the $w$-plane whose boundary is an
analytic Jordan curve $\Gamma$.\ \ By the Osgood--Carath\'eodory
theorem, $f$ extends continuously to $\Gamma$~\cite{henrici}.
We say that $\Omega$ {\em contains a finger of length $L>0$}
if there is a rectangular channel of width~$1$ defined by a pair
of parallel line segments of length $L$, disjoint from $\Omega$,
such that $\Omega$ extends all the way through the channel with
parts of $\Omega$ lying outside both ends.  Specifically, we
consider a point $a\in\Omega$ lying outside the channel at one
end and the nonempty portion $\Gex$ of $\Gamma$ lying outside
the channel at the other end, which we call the ``exit arc,''
as shown in Figure~\ref{finger}.  We denote the ends of the
channel by~$A$ and~$B$.  (If a part of $\Omega$ extends beyond
$B$ and then bends around to cross $B$ again and reenter the
rectangle, it is still ``beyond $B$'' in a topological sense
and this is how we interpret it and its portion of the
boundary $\Gamma$ throughout the discussion below.)

\begin{figure}[h]
\begin{center}
\vskip .2in
\includegraphics[scale=.75]{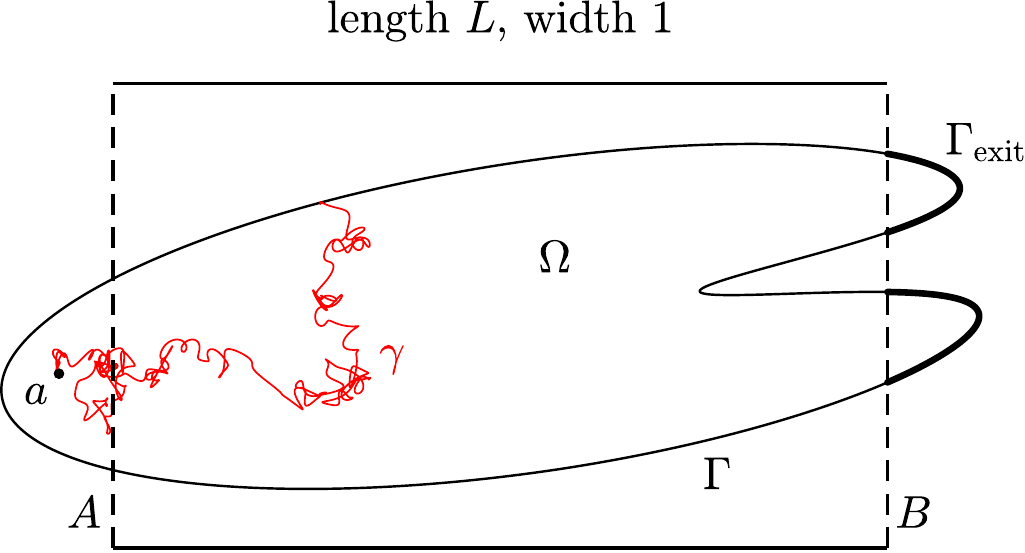}
\end{center}
\caption{\label{finger}A region $\Omega = f(D)$
with a {\em finger of length $L$} and a
Brownian path $\gamma$ starting at $a$ and ending when it hits
the boundary $\Gamma$.  The probability of
a path getting all the way through the
finger to first hit the boundary at a point beyond it
shrinks exponentially with $L$ at the rate 
$\exp(-\pi L)$.  This provides one explanation of the
exponential distortions characteristic of conformal maps, which
render rational approximations much more efficient
than polynomials.}
\end{figure}

Let $\omega\in (0,1)$ denote the harmonic measure of $\Gex$
at $a$ with respect to $\Omega$~\cite{ahl,gm,mp}.  This can be
interpreted as the probability that a Brownian path in $\Omega$
starting at $a$ first hits $\Gamma$ somewhere along $\Gex$,
and it is equal to the value $u(a)$ of the harmonic function
$u$ in $\Omega$ that takes the boundary values $1$ along
$\Gex$ and $0$ elsewhere on $\Gamma$.  Our first estimate
relates the geometry of $\Omega$ to the size of~$\omega$.
According to Garnett and Marshell~\cite[chap.~IV]{gm}, bounds
of this kind go back to Ahlfors in 1930~\cite{ahl30}, and in the conformal
mapping literature, related results have been derived by
authors including de Lillo, Dubiner, Pfaltzgraff, Pommerenke,
Wegmann, and Zemach~\cite{delillo,dp,pfluger,pomm,wegmann};
a survey is given in section~3 of~\cite{delillo}.  The most
general treatments rely on the notions of {\em extremal length}
and {\em extremal distance} from the theory of conformal
invariants~\cite{ahl,gm}, but our aim here is to state results
as simple as possible that contain the key factor $e^{-\pi L}$.
The following result is close to Theorem~6.1 of~\cite{gm}, which
has a constant $C=16$ and certain other differences.  The reason
for writing the bound in terms of the quotient $C/2\pi$ is that
it simplifies the formulas of the subsequent three theorems.

\begin{theorem}
Let\/ $\Omega$ contain a finger of length $L>1$, and let
$a\in\Omega$ be a point outside the finger on one end.  The
harmonic measure at\/ $a$ of the portion $\Gex$ of\/ $\Gamma$ beyond
the other end satisfies
\begin{equation} 
\omega < (C/2\pi)d\kern .6pt e^{-\pi L},
\label{fprimebound}
\end{equation} 
where $C$ is a constant independent of\/ $\Omega$ and $L$,
and $d$ is the length of the segment(s) $B\cap \Omega$.
Numerical computations indicate that a suitable
value is $C \approx 14.7$.
\end{theorem}

\begin{proof}
We argue by a sequence of inequalities $\omega \le  \omega_1
\le \omega_2 \le \omega_3 \le \omega_4$, followed by further
inequalities to be explained, with a Schwarz--Christoffel
conformal map at the end.  The sequence is summarized in
Figure~\ref{figseq}.  In each image, the point marked on the
left is $a$ and the thickly marked part of the boundary on the
right is $\Gex$.  From step to step, we sometimes adjust $a$
or $\Gex$.)

\begin{figure}
\vskip .3in
\begin{center}
\includegraphics[scale=.84]{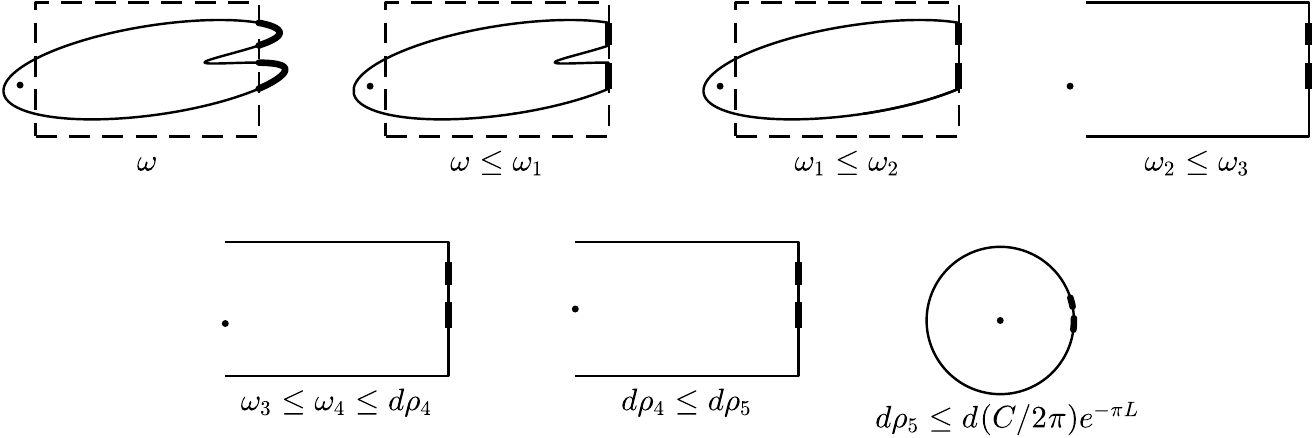}
\end{center}
\caption{\label{figseq}Sequence of steps for the proof
of Theorem 2, as explained in the text.
In each diagram $\omega_j$ denotes the
harmonic measure at the point $a$ on the left
of the exit arc $\Gex$ on the right with respect to the indicated
domain.}
\end{figure}

For the first step, from $\omega$ to $\omega_1$, we truncate
any portions of the finger that protrude beyond $B$, so that
$\Gex$ is now a subset of $B$.  By a principle of monotonicity of
harmonic measure, this gives $\omega\le \omega_1$.  Intuitively,
any Brownian path that reaches the boundary arcs of the first
image must have passed through one of the boundary arcs of
the second; thus the latter event must be at least as likely.
Note that the length of $\Gex$ is now exactly $d$.

Second, we expand any arcs of $\Gamma$ lying between arcs
of $\Gex$ to the boundary segment $B$, giving $\omega_1\le
\omega_2$ again by monotonicity.  Intuitively, enlarging the
domain means that more Brownian paths are available to reach
$\Gex$ before first hitting the boundary somewhere else.

Third, we expand $\Omega$ to be the whole exterior of the
three-sided open rectangle (a doubly-connected region in the
plane), giving $\omega_2 \le \omega_3$ by monotonicity once more.
From here on, the domain boundary consists of six segments,
three on the inside of the rectangle and three on the outside.
(Logically, this step could have been combined with the last one,
but for clarity it seems helpful to describe them separately.)

Fourth, any path from $a$ to the exit arc must have a point at
which it first touches the entry segment $A$ of the rectangle.
Some such points will correspond to greater probabilities than
others of eventually reaching $\Gex$.  We adjust $a$ to be a
point on $A$ that is optimal in this respect, giving $\omega_3
\le \omega_4$.

We have now reached the problem of estimating the harmonic
measure $\omega_4$ of a given segment or set of segments of
total length $d$ on the inside edge $B$ of the closed end of
an $L\times 1$ open-ended rectangle at a certain point $a$
on the open end.  Now at each point along $B$, there will be a
{\em harmonic measure density} for the point $a$, a continuous
function along $B$ taking positive values in the interior of $B$
and going to zero at the endpoints because of the corners.
Let $\rho_4$ denote the maximum value of this density.
Since $\Gex$ has length $d$, we have $\omega_4 \le d \rho_4$.

We next note that the number $\rho_4$ depends only on $L$
(which is fixed) and the point $a$ along the entry segment $A$.
Its maximal value for a given $L$ will occur when $a$ is at the
center of $A$, and it will be attained at the midpoint of $B$,
which we may call $b$.  We denote the value of the harmonic
measure density for this choice of $a$ by $\rho_5$, implying
$d\rho_4 \le d\rho_5$.

At this point we have a fully prescribed geometry,
an open-ended $L\times 1$ rectangle for which we
want the harmonic measure density $\rho_5$ of the
midpoint~$b$ of the inner closed end at the midpoint~$a$
of the open end.
This number is conformally invariant, so we can determine it by 
a Schwarz-Christoffel conformal map $g$ of the unit disk to the exterior
of the six-segment open rectangle boundary.  Specifically,
if $g(0) = a$ and $g(1) = b$,
then $\rho_5 = (\twopi |g'(1)|)^{-1}$.
As $L\to\infty$, it is well known that $\rho_5$
will decrease at a rate proportional to $e^{-\pi L}$  for a map with
a channel of this kind (see any of the references listed above
in connection with the crowding phenomenon, or for a beautiful
analysis of a special case, chapter 10 of~\cite{challenge}).
This establishes the theorem for some constant~$C$, and we have
obtained the estimate $C\approx 14.7$ by numerical computations
with the Schwarz--Christoffel Toolbox~\cite{toolbox}.
\end{proof}

Our next result asserts that if $\Omega$ contains a slender
finger, then $|f'(z)|$ must be exponentially large for some $z$
on the unit circle.

\begin{theorem}
Let\/ $\Omega = f(D)$ contain a finger of length $L>1$, and let
$a=f(0)$ be a point outside the finger on one end. 
Then there is a point $z$ with $|z|=1$
for which $|f'(z)| > C^{-1} e^{\pi L}$,
where $C$ is the same constant as in Theorem~$2$.
\end{theorem}

\begin{proof}
Under the assumptions of Theorem~2, the arc length of
$\Gex$ is at least~$d$ and its harmonic measure at $a$ is
$<(C/2\pi)d\kern .6pt e^{-\pi L}$.  Dividing these quantities,
we see that the average harmonic measure density along $\Gex$ is
$< (C/2\pi) e^{-\pi L}$.  Thus there are points along
$\Gex$ for which the harmonic measure density is
$< (C/2\pi) e^{-\pi L}$.  If $z$ is the preimage of such a
point under $f$, then by conformal invariance,
$|f'(z)| > C^{-1} e^{\pi L}$, as required.
\end{proof}

Theorems 2 and 3 have identified the crucial exponential factor $e^{\pi
L}$.  We now apply this result to derive two consequences with
implications for numerical approximation of conformal maps.

Define the {\em radius of univalence} of $f$, denoted by $r$,
as the supremum of all radii of open disks about $z=0$ to which
$f$ can be continued to a univalent (i.e., one-to-one) analytic
function.  Our first consequence of Theorems~2 and~3 asserts that
if $f$ has a finger of length $L$, then $r$ can be no bigger than
$1 + O(e^{-\pi L})$.  In typical cases, what is going on here is
that if $\Omega$ contains a finger, then $f$ has singularities
exponentially close to the unit circle.  The precise situation
is that $f$ need not actually have singularities (after all,
$f$ can be approximated by polynomials), but at least it must
lose univalence.

\begin{theorem}
Let\/ $\Omega = f(D)$ contain a finger of length $L>1$, let
$a=f(0)$ be a point outside the finger on one end,
and set $R = \sup_{w\in\Gamma} |w-f(0)|$.
Then the radius of univalence of\/ $f$ satisfies
\begin{equation}
r < 1 + 4 R \kern .4pt Ce^{-\pi L},
\label{rbound}
\end{equation}
where $C$ is the same constant as in Theorem\/~$2$, assuming
$4R\kern .4pt Ce^{-\pi L} < 1$.
\end{theorem}

\begin{proof}
Among the results of univalent function theory are certain
{\em distortion theorems}~\cite{ahl,gm}.  Let
$h$ be a univalent function in the unit disk
with $h(0) = 0$ and $h'(0) = 1$.  Then it is shown in
Theorem I.4.5 of~\cite{gm} that
for any $z$ with $|z|<1$,
\begin{equation}
|h'(z)| \le {(1+|z|)\kern .4pt |h(z)|\over |z|\kern .4pt (1-|z|)}.
\label{dist}
\end{equation}
To apply this to our function $f$ with radius of
univalence $r$, define
\begin{equation} 
h(z) =  {f(rz) - f(0) \over r f'(0)},
\label{hdef}
\end{equation}
which implies $h'(z) =  f'(rz)/f'(0)$.
Then applying (\ref{dist}) at a point $z$ with $|z|=r^{-1}$, so that $rz$
is on the unit circle, gives
\begin{equation}
\left | {f'(rz) \over f'(0)}\right| \le
{(1+r^{-1})\kern .4pt |h(z)|\over r^{-1}-r^{-2}},
\label{distf}
\end{equation}
or by (\ref{hdef}),
\begin{equation}
|f'(rz)| \le {(1+r^{-1})\kern .4pt |f(rz)-f(0)|\over 1-r^{-1}}.
\label{distf2}
\end{equation}
Since $r>1$ and $|f(rz)-f(0)|\le R$, this gives us
\begin{equation}
|f'(rz)| \le { 2R\over 1-r^{-1}}.
\label{distf3}
\end{equation}
Now assume that $rz$ is one of the points for which 
Theorem~3 gives the bound $|f'(rz)| > M = C^{-1}e^{\pi L}$.  Then
(\ref{distf3}) becomes $M < 2R/(1-r^{-1})$,
which is equivalent to
\begin{equation}
r < {1\over 1 - 2R/M}
\label{neweq}
\end{equation}
since the assumptions $4R\kern .4pt Ce^{-\pi L} < 1$ and $L>1$ imply
that $2R/M \le 1/2$, so that $1-2R/M$ is positive.
Since $(1-\varepsilon)^{-1}\le
1+2\kern .3pt \varepsilon$ whenever $\varepsilon \le 1/2$, (\ref{neweq})
implies $r<1+4\kern .3pt R/M = 1+4\kern .3pt R\kern .4pt
Ce^{-\pi L}$ as required.
\end{proof}

Our final consequence of Theorem 2 is that if a polynomial is
a good approximation to the conformal map $f$, then its degree
must be on the order of $e^{\pi L}$ or larger.  For a finger
of length $6$, say, one must expect polynomial degrees in
the millions.  In the following statement, the ``end segment
$B$'' refers to the same end of the rectangle illustrated in
Figure~\ref{finger} and discussed in the proof of Theorem~2.
Note that the definitions of $R$ here and in the last theorem
are slightly different.  For simplicity, our statement assumes
that $\Gex$ consists of a single arc.

\begin{theorem}
Let\/ $\Omega = f(D)$ contain a finger of length $L>1$, let
$a=f(0)$ be a point outside the finger on one end,
and set $R = \sup_{w\in\Gamma} |w| \ge 1$.
Assume the protrusion of $\Omega$ beyond the other
end segment $B$ delimits just a
single segment of\/ $B$ of length $d$.
If\/ $p$ is a polynomial satisfying 
$\|f-p\| \le d/3$ on the unit disk, then the degree\/ $n$ of\/ $p$
satisfies
\begin{equation}
n > {e^{\pi L}\over 4\kern .3pt  R\kern .5pt C},
\label{bound4}
\end{equation}
where $C$ is the same constant as in Theorem~$2$.
\end{theorem}

\begin{proof}
The assumptions imply $\|\kern .3pt p\| \le R + d/3$ on
the unit disk, so by an inequality due to
	Bernstein~\cite[Cor.\ (6,4)]{marden},
$\|\kern .3pt p'(z)\| \le n (R+d/3)$.
On the other hand, as we shall explain in the next paragraph,
the assumptions together with Theorem~2 imply that for
some $z$ with $|z|=1$,
\begin{equation}
|\kern .3pt p'(z)|\ge C^{-1} e^{\pi L}/ 3.
\label{pbound}
\end{equation}
Combining these bounds gives $n > C^{-1} e^{\pi L}/(3(R+d/3))$.
With $d\le 1$, $R\ge 1$, and $L\ge 1$, this inequality
implies (\ref{bound4}).

To establish (\ref{pbound}), let $w_1 = f(z_1)$ and $w_2 =
f(z_2)$ be the endpoints of the segment of $B$ indicated in the
theorem statement, with $|w_2-w_1| = d$ and $|z_1|=|z_2|=1$.
Since $\|f-p\|\le d/3$, we have $|\kern .5pt p(z_2)-
p(z_1)| \ge d/3$.  To justify (\ref{pbound}), it suffices to
show $|\arg(z_2)-\arg(z_1)| \le d \kern .6pt Ce^{-\pi L}$.
Since $|\arg(z_2)-\arg(z_1)|$ is $\twopi$ times the harmonic
measure at $f(0)$ of the protrusion of $\Gamma$, this follows
from Theorem~2.  \end{proof}

\section{Conformal maps of regions with analytic boundaries}

\begin{figure}
\begin{center}
\vskip .09in
\includegraphics[scale=0.907]{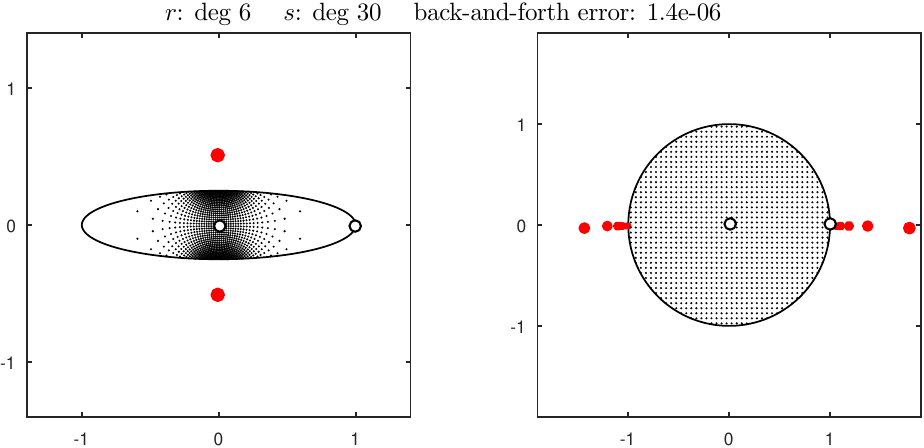}
\vskip .15in
\includegraphics[scale=0.907]{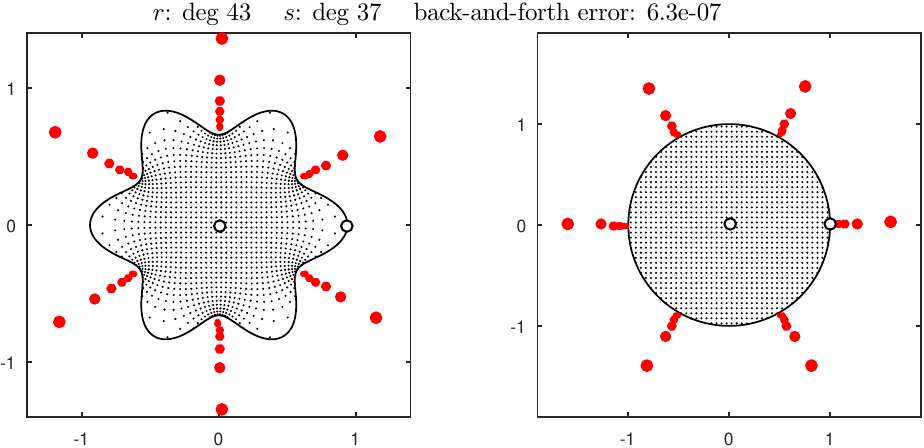}
\vskip .15in
\includegraphics[scale=0.907]{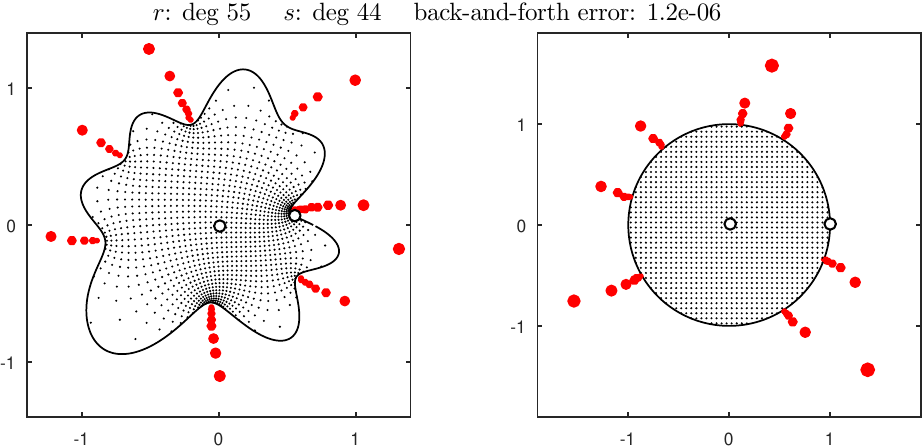}
\end{center}
\caption{\label{kerzman} Maps of three smooth domains to the unit
disk computed with the Kerzman--Stein integral equation, following
Caldwell, Li, and Greenbaum~{\rm \cite{clg,kerzstein,kerztrum}}.
Although the boundaries are analytic, these maps have singularities
nearby, as explained by the theorems of the
last section, and these are reflected in the poles of the adaptively
determined rational
approximations.  Specifically, the distances of the closest poles
to the unit disk in the plots of the second column are
about $0.0003$, $0.03$, and $0.007$,
respectively.}
\end{figure}

The prediction implicit in the last section is that even for
regions with analytic boundaries, rational approximations of
conformal maps will tend to have poles clustering nearby.
This prediction is borne out by Figures~\ref{kerzman}
and~\ref{twomore}, which show conformal maps of five regions with
analytic boundaries.  Qualitatively speaking, these images look
little different from those of section 3.  The computations were
done by a numerical implementation of the Kerzman--Stein integral
equation in Chebfun using codes developed by Caldwell, Li, and
Greenbaum~\cite{clg,kerzstein,kerztrum}, with a discretization by
800 points along the boundary.  Unlike the Schwarz--Christoffel
formula, the Kerzman--Stein equation constructs the mapping~$f$
in the direction from the problem domain to the unit disk.
As before, one of the attractions of rational approximations
is that they give representations of $f^{-1}$ as well as $f$.

The images in Figure~\ref{kerzman} show an ellipse with
semiaxes $1$ and $0.25$, a rounded ``snowflake'' defined in
polar coordinates by $r = 0.8 + 0.14\cos(6\kern .3pt \theta)$,
and a shape with a smooth random boundary defined in Chebfun
by setting $r(\theta)$ equal to
\verb|0.8+0.16*randnfun(0.6,[0 2*pi],'trig')|~\cite{randnfun}.
In each case the poles of the AAA rational approximation $r\approx
f$ are shown in the left panel and those of $s\approx f^{-1}$ are
shown in the right panel.  The degrees of $r$ and $s$ are listed
in the titles along with what we call the back-and-forth error,
defined as the maximum of $|w_j - r(s(w_j))|$ over the grid of
points $w_j$ plotted in the disk.  For these computations we
loosened the AAA tolerance from $10^{-7}$ to $10^{-6}$.

\begin{figure}
\begin{center}
\vskip .09in
\includegraphics[scale=0.907]{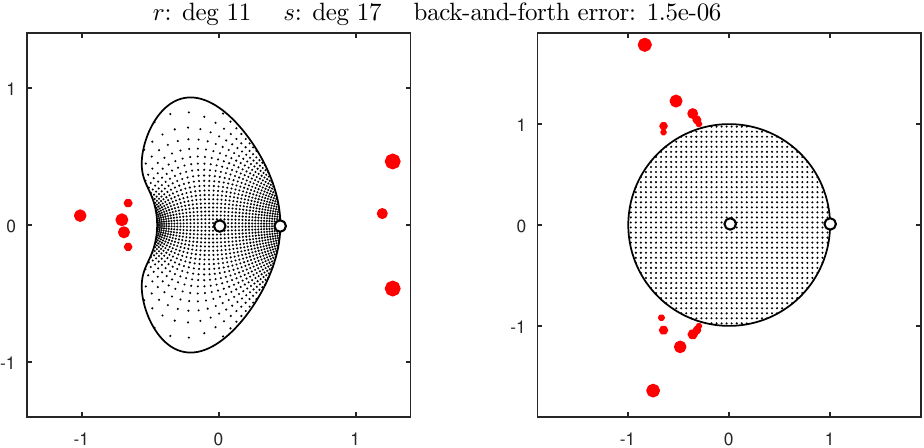}
\vskip .15in
\includegraphics[scale=0.907]{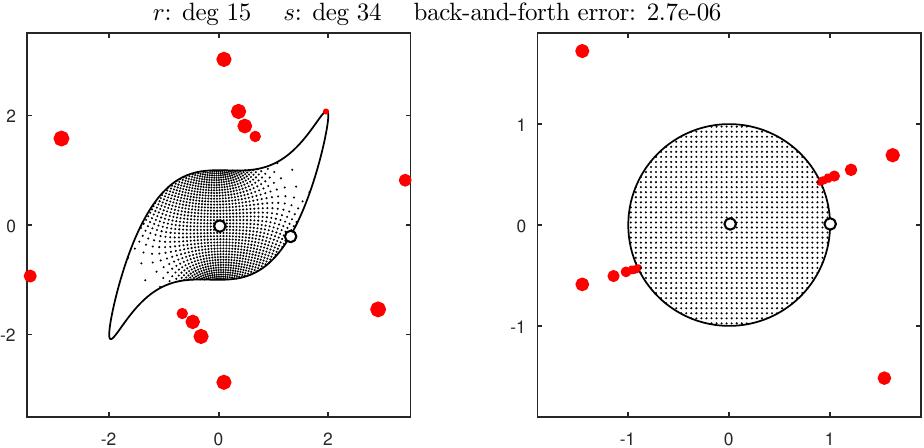}
\end{center}
\caption{\label{twomore} Maps of two more smooth domains.  The
``bean'' originates with Reichel in $1981$~{\rm\cite{reichel}} and
the ``blade'' with Ellacott in $1978$~{\rm\cite{ellacott}}.  The
poles of the rational approximations match only roughly the poles and
branch points of the exact maps.  The distances of the closest poles
to the unit disk in the second column are
about $0.05$ and $1.4\times 10^{-6}$.}
\end{figure}

As a further indication of the speed of rational representations
of conformal maps, suppose that for the region in the bottom
row of Figure~\ref{kerzman}, instead of a grid of about $10^3$
points in the disk (the exact number is 1264), we increase the
number to about $10^6$ and again compute $|w_j-r(s(w_j))|$,
thus applying a degree 55 rational function and also a degree 44
rational function one million times each.  On our laptop running
MATLAB, the total time is 1.7 seconds.  The maximum back-and-forth
error is $1.7\times 10^{-6}$, and the rms back-and-forth error
is $8.7\times 10^{-8}$.

Figure~\ref{twomore} shows two conformal maps
of regions with smooth boundaries taken from the classical
literature of numerical conformal mapping.  The
``bean'' in the first row, defined by the boundary map
$$
0.45\cos\theta + 0.225\cos(2\theta)-0.225
+ i\kern .5pt [\kern .3pt
0.7875\sin\theta + 0.225\sin(2\theta)-0.045\sin(4\theta)]
$$
for $\theta\in [\kern .3pt 0,2\pi]$,
was introduced by Reichel in the technical report
version of~\cite{reichel}
and considered also by Papamichael, et al.\ in~\cite{pwh}.
The ``blade'' in the second row, defined by the boundary map
$$
2\cos\theta + i(\sin\theta+ 2\cos^3\theta),
$$
was introduced by Ellacott~\cite{ellacott} and considered also
in~\cite{pwh,reichel}.  Both of these mapping problems were
discussed in~\cite{pwh} as examples in which the structure of
nearby poles and other singularities might be analyzed and exploited for
numerical approximation, just the opposite of the point of
view of the present paper.  The exact bean map has poles at
$\zeta_1 \approx -0.650$ and $\zeta_2\approx 1.311$, but the
rational approximation $r$ does not approximate these closely.
On the left we see five poles, not one, with clear asymmetry
about the real axis; as the authors of~\cite{pwh} point out,
the map additionally has square root branch points at $-0.566
\pm 0.068\kern .5pt i$.  As for $\zeta_2$, to the right of the
bean, the approximation $r$ shows three poles here,
not a single pole.  The pole near the real axis lies at
$1.200 + 0.079\kern .5pt i$.

The example of the blade region likewise illustrates that
exact singularities are only a rough guide to what may be
effective for rational approximation.  Here Papamichael, et
al.\ show that the exact map has simple poles at $\pm 2.885
\mp 1.584\kern .5pt i$, and these are matched to two digits
of accuracy by poles of $r$ at $2.908 - 1.557\kern .5pt i$
and $-2.889 +1.571\kern .5pt i$.  As with the bean region,
on the other hand, the portion of the figure with a concave
boundary segment is complicated by simple poles at $\pm 0.455
\pm 1.902\kern .5pt i$ with square root branch points nearby.
These are matched only roughly by poles of~$r$ at $0.475 +
1.813\kern .5pt i$ and $-0.475 - 1.787\kern .5pt i$.

\section{\label{numerics}Accuracy}
In the numerical experiments of Sections 3 and 5,
our rational approximations readily attain 7--8 digits of accuracy
almost everywhere in the domains, but fall to 2--4 digits close
to certain boundary points.  Let us call this the ``corner
effect'' (though domains with smooth boundaries are not immune).
This seems disturbing, but
as we shall now explain, it is actually a limitation not of
our rational approximation method, but of the accuracy of the
conformal mapping data we have been able to work with.

\begin{figure}
\begin{center}
\vskip .2in
\includegraphics[scale=.722]{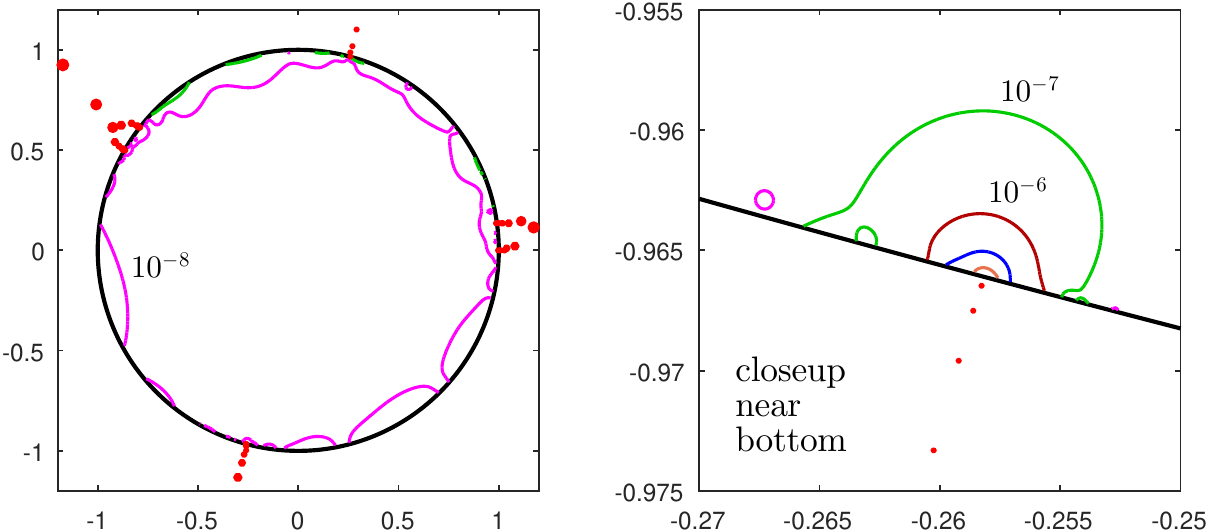}
\end{center}
\caption{\label{contplot}Error contours $|f(z)-r(z)| =
10^{-8},10^{-7},10^{-6},10^{-5},10^{-4}$ for the conformal map of
Figure\/~{\rm \ref{fig1_L}} of the unit disk onto an L-shaped
region.  On the left, mainly just the $10^{-8}$ contour is visible,
indicating that the error is ${<}\kern 1pt 10^{-8}$ in $90\%$ of the
disk and ${<}\kern 1pt 10^{-7}$ in most of the remainder.  Very
close to the vertices, however, the accuracy degrades, as shown by
the closeup on the right near the prevertex that maps to $w = 0$.
On this scale the $10^{-8}$ contour is not visible (apart from
a small bubble on the left) and we see
contours $10^{-7}$, $10^{-6}$, $10^{-5}$, and $10^{-4}$ as the
singular point is approached.  Thus, in a portion of
the disk which has an area on the order of\/ $10^{-6}$, the accuracy
falls to\/ $3$--$4$ digits.  This explains the max error
value $2.6\times 10^{-3}$ in the upper right panel of
Figure~$\ref{fig1_L}$.}
\end{figure}

First we present a figure to illustrate the effect.
Figure~\ref{fig1_L} showed a map~$f$ of the unit disk onto an
L-shaped region, and the upper-left image of that figure plotted
the unit disk together with the poles of the rational approximation
$r\approx f$.  For the same problem, Figure~\ref{contplot} plots
contours of the error $|f(z)-r(z)|$.  As explained in the caption,
the contours confirm that the error is smaller than $10^{-7}$
almost everywhere in the disk but much worse near the prevertices.

The corner effect does not result from inherent limitations in the
accuracy of rational approximations, since the root-exponential
convergence is very fast.
Our examples have shown that with ${\approx}\kern
1pt 10$ poles near each singularity, rational functions are
well capable of $10^{-3}$ accuracy, and there is no reason of
approximation theory why this could not be improved to $10^{-6}$
with ${\approx}\kern 1pt 40$ poles near each singularity.

The corner effect also does not result from limitations of the AAA
algorithm.  The algorithm does not produce an optimal approximant,
but our experience shows that it reliably comes within one or
two orders of magnitude of optimality on the discrete point set
it works with. 

Instead, the corner effect is a consequence of {\em inaccurate
boundary data.} As we now explain, to achieve an approximation
$r$ with $\|r-f\| = O(\varepsilon)$ (the maximum norm on the
domain) everywhere, it is necessary to sample the boundary map at
distances $O(\varepsilon)$ from prevertices or vertices with an
accuracy of $o(\varepsilon)$.  (The ``$O$'' and ``$o$'' symbols
here are heuristic, not precisely defined.)  Partly because of
the extreme ill-conditioning of conformal maps near corners,
as well as in other contexts, and partly because of additional
problems of accuracy of the SC Toolbox very near corners, this
is hard to achieve when $\varepsilon$ is much less than $10^{-4}$.

\begin{figure}
\begin{center}
\vskip .5in
\includegraphics[scale=.722]{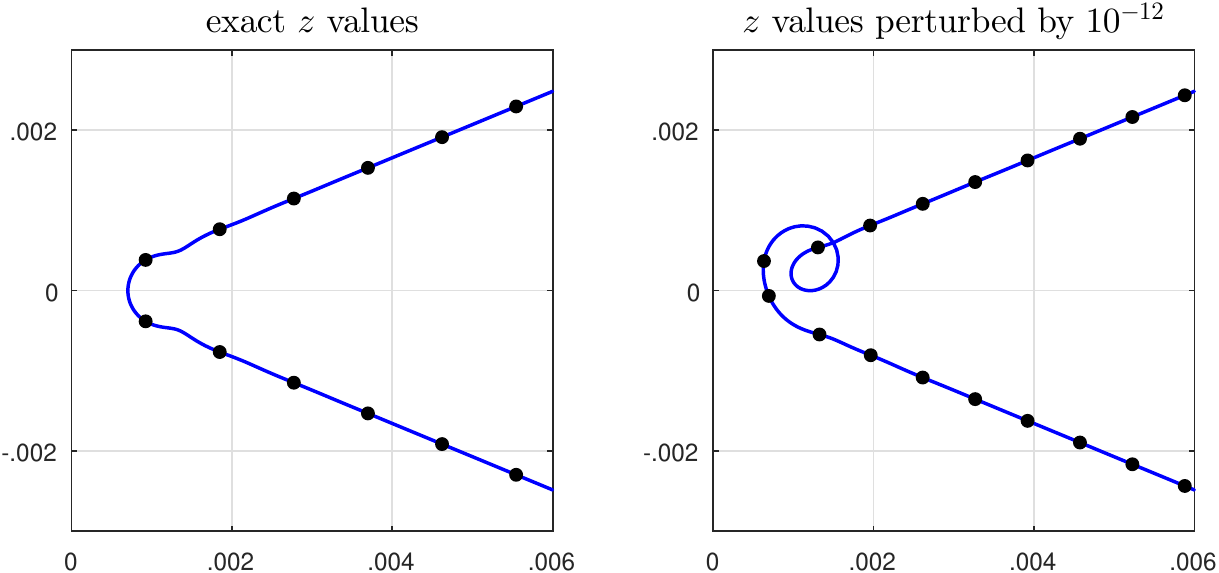}
\end{center}
\caption{\label{accfig} On the left, successful AAA approximation
near a\/ $z^{1/4}$ singularity.  On the right, the values $z_j$ have
been perturbed on the scale of $10^{-12}$, and the interpolant loses
accuracy near the corner.  Moreover, it now has poles inside the
approximation domain, implying that even though the error is still
very small almost everywhere in the domain, its maximum is\/
$\infty$.}
\end{figure}

Figure~\ref{accfig} illustrates the situation.  The problem
considered here is the map $w = f(z) = z^{1/4}$ of the right
half-plane onto the wedge of angle $\pi/4$ about the positive
real axis.  In the left panel of the figure, this function has been
approximated by the AAA algorithm with tolerance $10^{-7}$ based
on data at $200$ points $z_j$ fixed so that the images $f(z_j)$
are equispaced along the two sides of the wedge near $w=0$ out to
a distance~$0.1$.   Thus the points $z_j$ themselves lie in the
imaginary interval $[-10^{-4}\kern .5pt i, 10^{-4}\kern .5pt i]$,
and the two values closest to $0$ are $\pm 10^{-12}\kern .5pt i$.
The black dots in the figure are the data, and the blue curve
is the AAA approximant.  The behavior we see represents just
what one would hope for.  The curve closely matches the data and
interpolates smoothly in-between.

This left panel of Figure~\ref{accfig} reflects a rational
approximation with a maximal error on the order of $10^{-4}$.
To improve this to a level $O(\varepsilon)$ for some $\varepsilon
\ll 10^{-4}$, it is clear that the mesh will have to be refined
so that the data values near $w=0$ have spacing $O(\varepsilon)$.
This means the corresponding $z$ values will have minimal spacing
$O(\varepsilon^4)$.  And here is where we see the difficulty in
accuracy near corners.  In 16-digit floating point arithmetic,
spacings on the order of $\varepsilon^4$ will not be accurately
computable for $\varepsilon\ll 1$.  To illustrate this effect,
the right panel shows the same computation again, except that
instead of data values $f(z_j)$, the values are taken to be
$f(\tilde z_j)$, where $\tilde z_j$ is $z_j$ plus a random complex
perturbation of order $10^{-12}$.  Now the rational approximation,
while still excellent away from the vertex, has quite erroneous
behavior nearby.

Based on these curves alone, one might guess that the
approximation in the right panel of Figure~\ref{accfig} was, say,
ten times less accurate than for the one on the left.  However,
the loop in the curve is a hint that something more fundamental
is going wrong.  In fact, this rational function has poles in
the approximation domain---three of them, as it happens, with
real parts on the order of $10^{-8}$ and residues on the order
of $10^{-7}$.  From the point of view of the original problem,
these poles are spurious; they have been introduced by the
non-analytic perturbation of the fitting data.  They will have
negligible effect away from the vertex, but near the vertex,
there is a portion of the approximation domain in which the error
is infinite.  One might not notice this in an application, but it
can hardly be regarded as acceptable.  

We believe that the inaccuracy of available boundary data, often
due to the ill-conditioning of the conformal map, is usually what
limits the accuracy of AAA approximants.  If one could work with
perfect data on grids clustered very finely near singularities,
the rational approximants would have no trouble, but when the
data are imperfect, one loses accuracy correspondingly.  The SC
Toolbox, unfortunately, seems to have additional difficulties
near vertices that make the situation somewhat worse than what
is inevitable due to ill-conditioning alone.

One might imagine that conformal maps of smooth domains, since
there are no singularities on the boundary, would be exempt
from these accuracy limitations.  However, as we showed in Section 4,
this is not the case since there are often singularities exponentially
close to the boundary.

\section{Conclusion}
In this paper we established theorems concerning the
power of rational functions in approximating conformal
maps near singularities (Theorem 1) and the prevalence of
singularities and weakness of polynomial approximations
for such maps (Theorems 2--5).  We then showed how effective
AAA rational approximations are in realizing this
potential to represent conformal maps far more efficiently
than is usual.  Their application is fast and easy to
apply, capturing singularities to good accuracy without the need
for any analysis.  In a few seconds
one typically gets a representation of both a conformal map $f$
and its inverse $f^{-1}$ that can be applied to map points back
and forth in microseconds on a laptop.  The main complication,
as with most computing with rational functions, it that it is
advisable to monitor whether spurious poles have appeared in the
region of approximation.  If so, they can usually be removed by
loosening the convergence tolerance, improving the accuracy of
the data, or refining the grid.

Our approximations are rational functions of type $(n,n)$,
typically with $10\le n\le 100$.  To get comparable accuracy with
polynomial approximations, one would often
need degrees in the millions.

Our discussion has concerned simply connected regions, but the
same methods apply to conformal maps of multiply connected
regions~\cite{bds}, and indeed, to analytic and meromorphic
functions more generally.

\begin{acknowledgements}
This paper originated in stimulating discussions with Anne
Greenbaum and Trevor Caldwell about their computations with
the Kerzman--Stein integral equation, and Grady Wright gave
key assistance in a Chebfun implementation.  The heart of the
paper is Schwarz--Christoffel mapping, which is made numerically
possible by Toby Driscoll's marvelous SC Toolbox.  Driscoll,
and Yuji Nakatsukasa offered helpful advice along the way, and
the suggestions of Dmitry Belyaev, Chris Bishop, and Tom DeLillo were crucial
for developing the theorems of Section~4.  Among other
things, Belyaev caught an error in an early version of
Theorems~2 and~3 and Bishop pointed us to
Theorem~6.1 of~\cite{gm} and
proposed the idea of Theorem 4.  Much of 
this article was written during an extremely
enjoyable 2017--18 sabbatical visit by the second
author to the Laboratoire de l'Informatique du Parall\'elisme at
ENS Lyon hosted by Nicolas Brisebarre, Jean-Michel Muller, and
Bruno Salvy.
\end{acknowledgements}


\begin{thebibliography}{}

\bibitem{ahl30}
Ahlfors, L.: Untersuchungen zur Theorie der konformen
Abbildung und der ganzen Funktionen, Dr. der Finnischen Literaturges. (1930)

\bibitem{ahl}
Ahlfors, L.: Conformal Invariants: Topics
in Geometric Function Theory, McGraw-Hill (1973)

\bibitem{bds}
Badreddine, M., DeLillo, T. K., Sahraei, S.:
A comparison of some numerical conformal mapping methods
for simply and multiply connected domains, Discrete and
Continuous Dynamical Systems Series B
{\bf 24}, 55--82 (2019)

\bibitem{bgm}
Baker, Jr., G. A., Graves-Morris, P.:
Pad\'e Approximants, 2nd ed., Cambridge U. Press (1996)

\bibitem{banjai}
Banjai, L.:
Revisiting the crowding phenomenon in Schwarz--Christoffel
mapping,
SIAM J. Sci.\ Comp.\ {\bf 30}, 618--636 (2008)

\bibitem{blm}
Beckermann, B., Labahn, G., Matos, A. C.:
On rational functions without Froissart doublets,
Numer.\ Math.\ {\bf 138}, 615--633 (2018)

\bibitem{challenge}
Bornemann, F.  Laurie, D., Wagon, S.  Waldvogel, J.:
The SIAM 100-Digit Challenge: A Study in High-Accuracy
Numerical Computing, SIAM (2004)

\bibitem{clg}
Caldwell, T., Li, K., Greenbaum, A.:
Numerical conformal mapping in Chebfun,
poster, Householder XX Symposium on Numerical Linear Algebra,
Blacksburg, VA (2017)

\bibitem{cmft}
Computational Methods and Function Theory, special
issue on numerical conformal mapping, vol.~11, no.~2,
375--787 (2012)

\bibitem{delillo}
DeLillo, T. K.:
The accuracy of numerical conformal mapping methods:
a survey of examples and results,
SIAM J. Numer.\ Anal.\ {\bf 31}, 788--812 (1994)

\bibitem{dp}
DeLillo, T. K., Pfaltzgraff, J. A.:
Extremal distance, harmonic measure and
numerical conformal mapping,
J. Comp.\ Appl.\ Math.\ {\bf 46}, 103--113 (1993)

\bibitem{toolbox}
Driscoll, T. A.:
Algorithm 756: A MATLAB toolbox for Schwarz--Christoffel
mapping,
ACM Trans.\ Math.\ Softw.\ {\bf 22}, 168--186 (1996);
see also \verb|www.math.udel.edu/~driscoll/SC/|

\bibitem{chebfun}
Driscoll, T. A., Hale, N., L. N. Trefethen, L. N.,
eds.: Chebfun User's Guide, Pafnuty Publications, Oxford (2014);
see also {\tt www.chebfun.org}

\bibitem{SCbook}
Driscoll, T. A., Trefethen, L. N.:
Schwarz--Christoffel Mapping, Cambridge U. Press (2002)

\bibitem{ellacott}
Ellacott, S. W.:
A technique for approximate conformal mapping,
in Multivariable Approximation (D. C. Handscomb, ed.),
Academic Press, London (1978)

\bibitem{randnfun}
Filip, S., Javeed, A., Trefethen, L. N.:
Smooth random functions, random ODEs, and Gaussian pocesses,
SIAM Rev., submitted

\bibitem{fntb}
Filip, S., Nakatsukasa, Y., Trefethen, L. N., Beckermann, B.:
Rational minimax approximation via adaptive barycentric
representations, SIAM J. Sci. Comp., submitted

\bibitem{gaier}
Gaier, D.:
Konstruktive Methoden der konformen Abbildung,
Springer (1964)

\bibitem{gm}
Garnett, J. B., Marshall, D. E.:
Harmonic Measure, Cambridge U. Press (2005)

\bibitem{gpt}
Gonnet, P., G\"uttel, S., Trefethen, L. N.:
Robust Pad\'e approximation via SVD,
SIAM Review {\bf 55}, 101--117 (2013)

\bibitem{gww}
Gordon, C., Webb, D., Wolpert, S.:
Isospectral plane domains and surfaces
via Riemann\-ian orbifolds,
Invent.\ Math.\ {\bf 110}, 1--22 (1992)

\bibitem{grad}
Gradshteyn, I. S., Ryzhik, I. M.:
Table of Integrals, Series, and Products,
5th ed., Academic Press (2014)

\bibitem{hqr}
Hakula, H., Quach, T., A. Rasila, A.:
Conjugate function method for numerical conformal
mappings, J. Comput.\ Appl.\ Math.\ {\bf 237}, 340--353 (2013)

\bibitem{henrici}
Henrici, P.: Applied and
Computational Complex Analysis. III, Wiley (1974)

\bibitem{kerzstein}
Kerzman, N., Stein, E. M.:
The Cauchy kernel, the Szeg\H o kernel, and the Riemann mapping
function, Math.\ Ann.\ {\bf 236}, 85--93 (1978)

\bibitem{kerztrum}
Kerzman, N., Trummer, M. R.:
Numerical conformal mapping via the Szeg\H o kernel,
J. Comp.\ Appl.\ Math.\ {\bf 14}, 111--123 (1986)

\bibitem{lehman}
Lehman, R. S.:
Development of the mapping function at an
analytic corner, Pacific
J. Math.\ {\bf 7}, 1437--1449 (1957)

\bibitem{marden}
Marden, M.:
Geometry of Polynomials,
Amer.\ Math.\ Soc.\ (1949)

\bibitem{mp}
M\"orters, P., Peres, Y.:
Brownian Motion, Cambridge, U. Press (2010)

\bibitem{AAA}
Nakatsukasa, Y., S\`ete, O., Trefethen, L. N.:
The AAA algorithm for rational approximation,
SIAM J. Sci.\ Comp.\ {\bf 40}, A1494--A1522 (2018)

\bibitem{newman}
Newman, D. J.:
Rational approximation to $|x|$,
Mich.\ Math.\ J.\ {\bf 11}, 11--14 (1964)

\bibitem{perron}
Perron, O.:
Die Lehre von den Kettenbr\"uchen,
Leipzig and Berlin (1913)

\bibitem{papstyl}
Papamichael, N., Stylianopoulos, N.:
Numerical Conformal Mapping: Domain Decomposition
and the Mapping of Quadrilaterals, World Scientific (2010)

\bibitem{pwh}
Papamichael, N., Warby, M. K., Hough, D. M.:
The treatment of corner and pole-type singularities
in numerical conformal mapping techniques,
J. Comp.\ Appl.\ Math.\ {\bf 14}, 163--191 (1986)

\bibitem{pfluger}
Pfluger, A.:
Extremall\"angen und Kapazit\"at, 
Comment.\ Math.\ Helv.\ {\bf 29}, 120--131 (1955)

\bibitem{pomm}
Pommerenke, Ch.: 
Boundary Behavior of Conformal Maps,
Springer-Verlag, Berlin, Heidelberg, New York (1992)

\bibitem{reichel}
Reichel, L.:
On polynomial approximation in the complex
plane with application to conformal mapping,
Math.\ Comp.\ {\bf 44}, 425--433 (1985).
An earlier technical report version
version with additional material
appeared as Report TRITA-NA-8102, Dept.\ of Numerical Analysis
and Computing Science, Royal Institute of Technology,
Stockholm (1981)

\bibitem{stahl}
Stahl, H.:
The convergence of Pad\'e approximants
to functions with branch points, J. Approx.\ Th.\ {\bf 91},
139--204 (1997)

\bibitem{stahlspurious}
Stahl, H.:
Spurious poles in Pad\'e approximation,
J. Comp.\ Appl.\ Math.\ {\bf 99}, 511--527 (1998)

\bibitem{stahl03}
Stahl, H. R.:
Best uniform rational
approximation of $x^\alpha$ on $[\kern .3pt 0,1\kern .2pt]$,
Acta Math.\ {\bf 190}, 241--306 (2003)

\bibitem{suetin}
Suetin, S. P.:
Distribution of the zeros of Pad\'e polynomials
and analytic continuation,
Russian Math.\ Surveys {\bf 70}, 901--951 (2015)

\bibitem{sc}
Trefethen, L. N.:
Numerical computation of the
Schwarz--Christoffel transformation, SIAM J.
Sci.\ Stat.\ Comp.\ {\bf 1}, 82--102 (1980)

\bibitem{atap}
Trefethen, L. N.: Approximation Theory and
Approximation Practice, SIAM, 2013.

\bibitem{tg}
Trefethen, L. N., Gutknecht, M. H.: Nonuniqueness
of best rational approximations on the unit disk,
J. Approx.\ Th.\ {\bf 39}, 275--288 (1983)

\bibitem{trap}
Trefethen, L. N., Weideman, J. A. C.:
The exponentially convergent trapezoidal rule,
SIAM Review {\bf 56}, 385--458 (2014)

\bibitem{wegmann}
Wegmann, R.:
Methods for numerical conformal mapping,
in Handbook of Complex Analysis: Geometric Function Theory, v.~2,
R. K\"uhnau, ed., Elsevier, 351--477 (2005)

\bibitem{fornberg}
Wright, G. B., Fornberg, B.:
Stable computations
with flat radial basis functions using vector-valued
rational approximations,
J. Comp.\ Phys.\ {\bf 331}, 137--156 (2017)

\end{thebibliography}
\end{document}